\documentclass[11pt,reqno]{amsart}

\usepackage{amssymb}
\usepackage{amsmath}
\usepackage{amsthm}
\usepackage{hyperref}

\numberwithin{equation}{section}

\newtheorem{theorem}{Theorem}[section]
\newtheorem{definition}[theorem]{Definition}
\newtheorem{proposition}[theorem]{Proposition}

\newtheorem{remark}[theorem]{Remark}

\begin{document}

\title{An orthogonality relation for a thin family of $GL(3)$ Maass forms}

\author[Jo\~{a}o Guerreiro]{Jo\~{a}o Guerreiro}

\date{\today}


\begin{abstract}
We prove an orthogonality relation for the Fourier-Whittaker coefficients of a thin family of $GL(3)$ Maass forms containing all self-dual forms. This is obtained by analysing the Kuznetsov trace formula on $GL(3)$ for a certain family of test functions. The method also yields Weyl's law for the same family of Maass forms.
\end{abstract}

\maketitle



\section{Introduction}

A well known fact about Dirichlet characters is the following orthogonality relation
\begin{equation}
\sum_{\chi \bmod q} \chi(n) \overline{\chi(m)} =
\begin{cases}
\phi(q), \text{ if } n \equiv m \bmod q \\
0, \text{ otherwise.}
\end{cases}
\end{equation}
for integers $m, n$ coprime to $q$, where the sum on the left is over all characters$\pmod q$. Since Dirichlet characters can be viewed as automorphic representations of $GL(1,\mathbb{A}_\mathbb{Q})$, this result can be interpreted as the simplest case of the orthogonality relation conjectured by Zhou \cite{zhou} concerning Fourier-Whittaker coefficients of Maass forms on the space $SL_n(\mathbb{Z}) \backslash SL_n(\mathbb{R}) / SO_n(\mathbb{R})$, $n \geq 2$. This orthogonality relation conjectured by Zhou was proved by Bruggeman \cite{bruggeman} in the case $n=2$ and by Goldfeld-Kontorovich \cite{goldfeld} and Blomer \cite{blomer} in the case $n=3$. Versions of this result have applications to the Sato-Tate problem for Hecke operators, both in the holomorphic \cite{cdf}, \cite{serre} and non-holomorphic setting \cite{sarnak}, \cite{zhou}, as well as to the problem of determining symmetry types of families of L-functions \cite{goldfeld} as introduced in the work of Katz-Sarnak \cite{katzsarnak}. In this paper, we prove an orthogonality relation for the Fourier-Whittaker coefficients of a thin family of Maass forms on $SL_3(\mathbb{Z})$ which contains all self-dual forms. The main tool used is the Kuznetsov trace formula for $GL(3)$ developed by Blomer \cite{blomer} and Goldfeld-Kontorovich \cite{goldfeld}. The same methods yield a weighted Weyl's law for the same family of Maass forms.

Weyl's law was first proved by Selberg \cite{selberg} for Maass forms on $SL_2(\mathbb{Z})$ using the Selberg trace formula. Miller \cite{miller} later showed Weyl's law for Maass forms on the space $SL_3(\mathbb{Z}) \backslash SL_3(\mathbb{R}) / SO_3(\mathbb{R})$. This result has since been obtained in more general settings in \cite{lapidmuller}, \cite{lindenstrauss}, \cite{muller}. The version of Weyl's law presented here tells us that the family of Maass forms being studied has zero density in the set of all Maass forms on $SL_3(\mathbb{Z})$. We also note that a lower bound Weyl's law for self-dual forms on $SL_3(\mathbb{Z})$ follows from the Gelbart-Jacquet lift \cite{gelbartjacquet} and Weyl's law for Maass forms on $SL_2(\mathbb{Z})$. This work is a step towards a proof of the Gelbart-Jacquet lift by isolating the contribution of symmetric square lifts from $GL(2)$ in the $GL(3)$ Kuznetsov trace formula, in the spirit of Langlands' ``Beyond Endoscopy" \cite{langlands}.

Let $\phi_j$ ($j=0,1,2,\dots$), where $\phi_0$ is a constant function, be a set of orthogonal $GL(3)$ Hecke-Maass forms with spectral parameters $\nu^{(j)} = \left(\nu^{(j)}_1, \nu^{(j)}_2\right)$, and define $\nu^{(j)}_3 = -\nu^{(j)}_1-\nu^{(j)}_2$. The parameters are normalized such that for tempered forms they are purely imaginary.
For a particular non-constant Maass form $\phi$ with spectral parameters $\nu = (\nu_1, \nu_2)$ define $\nu_3 = -\nu_1-\nu_2$, and define Langlands parameters 
\[\alpha_1 = \nu_1 - \nu_3, \ \alpha_2 = \nu_2 - \nu_1, \ \alpha_3 = \nu_3 - \nu_2. \] The Laplace eigenvalue of $\phi$ is given by

\[ \lambda_{\phi} = 1 - 3(\nu_1^2 + \nu_2^2 + \nu_3^2) = 1 - (\alpha_1^2 +\alpha_2^2 + \alpha_3^2). \]
Let \[ \mathcal{L}_j = \text{Res}_{s=1} L(s, \phi_j \times \tilde{\phi_j}). \]
For $T \gg 1$, $\nu = (\nu_1, \nu_2) \in \mathbb{C}^2$ (with $\nu_3 = -\nu_1 -\nu_2$) and $R > 0$, to be fixed later, we define

\begin{equation}
\label{defh}
h_{T,R}(\nu) = \left( T^{\alpha_1^2/R^2} + T^{\alpha_2^2/R^2} + T^{\alpha_3^2/R^2} \right)^2 e^{2(\alpha_1^2 + \alpha_2^2 + \alpha_3^2)/T^2} \frac{\left(\prod \limits_{j = 1}^3 \Gamma\left(\frac{2+R+3\nu_j}{4} \right) \Gamma\left(\frac{2+R-3\nu_j}{4} \right)\right)^2}{\prod \limits_{j=1}^3 \Gamma\left(\frac{1+3\nu_j}{2} \right) \Gamma\left(\frac{1-3\nu_j}{2} \right)}.
\end{equation} 

This function is essentially supported on the region where $|\nu_1|, |\nu_2|, |\nu_3| \ll T$.
Let $\phi$ be tempered and $|\nu_1|, |\nu_2|, |\nu_3| < T$. Note that in this situation, $h_{T,R}(\nu)$ is real valued and positive. We estimate this function in three regions.
If one of the $\alpha_i$ is equal to $0$, then

\begin{equation} c_R^{(1)} [(1+|\nu_1|)(1+|\nu_2|)(1+|\nu_3|)]^R \leq h_{T,R}(\nu) \leq c_R^{(2)} [(1+|\nu_1|)(1+|\nu_2|)(1+|\nu_3|)]^R, \end{equation} for some $c_R^{(2)} \geq c_R^{(1)} > 0$.
If one of the $|\alpha_i|$ is smaller than $R$, then

\begin{equation} h_{T,R}(\nu) \gg_R \ \frac{[(1+|\nu_1|)(1+|\nu_2|)(1+|\nu_3|)]^R}{T^2}. \end{equation}
If $|\alpha_1|, |\alpha_2|, |\alpha_3| > 2R^{3/2} $, then

\begin{equation} h_{T,R}(\nu) \ll \frac{1}{T^R}. \end{equation}




\begin{theorem}

\label{mainthm}

For $j=0,1,2,\dots,$ let $\phi_j$ be a set of orthogonal $GL(3)$ Hecke-Maass forms with spectral parameters $\nu^{(j)} = \left(\nu^{(j)}_1, \nu^{(j)}_2\right)$. Here $\phi_0$ is a constant function. Let $h_{T,R}$ be given as in (\ref{defh}). We have that for fixed $R > 50$ and some $c_R > 0$,

\begin{equation} \sum \limits_{j \geq 1} \frac{ \left|h_{T,R}(\nu^{(j)})\right|}{\mathcal{L}_j} = c_R \frac{T^{4+3R}}{\sqrt{\log T}} + O(T^{3+3R}), \ \ \ \ \ \ \ (T \rightarrow \infty). \end{equation}
Moreover, for fixed $\epsilon > 0$ and positive integers $m_1, m_2, n_1, n_2$, we have

\begin{equation}
\sum \limits_j A_j(m_1, m_2) \overline{A_j(n_1, n_2)} \frac{ \left|h_{T,R}(\nu^{(j)})\right|}{\mathcal{L}_j} =
\begin{cases}
\sum \limits_{j} \frac{ \left|h_{T,R}(\nu^{(j)})\right|}{\mathcal{L}_j} + \mathcal{O}_{R, \epsilon} \bigg((m_1 m_2 n_1 n_2)^6 \, T^{3R+3+\epsilon} \bigg), \text{ if } \substack{m_1 = n_1 \\ m_2 = n_2} \\
\mathcal{O}_{R, \epsilon} \bigg( (m_1 m_2 n_1 n_2)^6 \, T^{3R+3+\epsilon} \bigg), \text{ otherwise.}
\end{cases}
\end{equation}

\end{theorem}

\begin{remark}
Note that an analogous result in \cite{goldfeld} for all $GL(3)$ Hecke-Maass forms with spectral parameters $|\nu_1|, |\nu_2|,|\nu_3| \ll T$ yields a main term of size $T^{5+3R}$. This implies that the family of $GL(3)$ Maass forms picked out by our test function $h_{T,R}$ is indeed significantly thinner (by a factor of $T\sqrt{\log T}$) than the family of all $GL(3)$ Hecke-Maass forms.
\end{remark}

\begin{remark}
The test functions $h_{T,R}$ appearing in this theorem are a product of three terms chosen with the following objectives: the first exponential term contributes with polynomial decay when all $|\alpha_1|, |\alpha_2|, |\alpha_3| \gg 0$ and exponential decay when all $|\alpha_1|, |\alpha_2|, |\alpha_3| \gg T^{\epsilon}$; the second exponential term contributes with exponential decay when one of $|\nu_i| > T^{1+\epsilon}$; the product of Gamma factors is already partially present in the Kuznetsov trace formula and its particular form is such that it has polynomial growth in $\nu$. We note that the methods presented here have also been carried out for a different family of test functions in \cite{goldfeld}.
It would be interesting to generalize \ref{mainthm} for a broader class of test functions.
\end{remark}

\begin{remark}
Blomer \cite{blomer} shows that the residues $\mathcal{L}_j$ are bounded by 

\begin{equation} 
\left((1+|\nu_1^{(j)}|)(1+|\nu_2^{(j)}|)(1+|\nu_3^{(j)}|)\right)^{-1} \ll \mathcal{L}_j \ll_{\epsilon} 
\left((1+|\nu_1^{(j)}|)(1+|\nu_2^{(j)}|)(1+|\nu_3^{(j)}|)\right)^{\epsilon} 
\end{equation}
and, conjecturally, the lower bound is expected to be 

\begin{equation} \label{residue}
\left(1+|\nu_1^{(j)}|)(1+|\nu_2^{(j)}|)(1+|\nu_3^{(j)}|)\right)^{-\epsilon} \ll_{\epsilon} \mathcal{L}_j.
\end{equation}

\end{remark}


\newpage

\section{$GL(3)$ Kuznetsov trace formula}

Following (\cite{blomer}, Proposition 4) and \cite{goldfeld}, the $GL(3)$ Kuznetsov trace formula takes the form

\begin{equation} \mathcal{C} + \mathcal{E}_{min} + \mathcal{E}_{max} = \Sigma_1 + \Sigma_{2a} + \Sigma_{2b} + \Sigma_3,\end{equation}
where all the quantities are explained below.

\begin{definition}
Let $\nu \in \mathbb{C}^2$ and $z \in \mathfrak{h}$ with Iwasawa coordinates

\[ z = x y = \left( \begin{array}{ccc}
1 & x_2 & x_3 \\
0 & 1 & x_1 \\
0 & 0 & 1 \end{array} \right) \left( \begin{array}{ccc}
y_1 y_2 & 0 & 0 \\
0 & y_1 & 0 \\
0 & 0 & 1 \end{array} \right).\]

Let
\[ I_{\nu}(z) = y_1^{\nu_1+2\nu_2} y_2^{2\nu_1 + \nu_2}, \ (z = xy \in \mathfrak{h}). \]
Define the Whittaker function

\begin{multline} W_{\nu}(z) = \pi^{-\nu_1-\nu_2} \Gamma \left( \frac{1+3\nu_1}{2} \right) \Gamma \left( \frac{1+3\nu_2}{2} \right) \Gamma \left( \frac{1+3\nu_1+3\nu_2}{2} \right) \\
\times \iiint \limits_{\mathbb{R}^3} I_{\nu} \left( \left(  \begin{array}{ccc}
 &  & 1 \\
 & 1 &  \\
1 &  &  \end{array} \right)
\left( \begin{array}{ccc}
1 & u_2 & u_3 \\
 & 1 & u_1 \\
 &  & 1 \end{array}
\right) \right) e(-u_1-u_2)  du_1 du_2 du_3,
\end{multline}
where $e(x) = e^{2\pi i x}$.

\end{definition}
Let $F: \mathbb{R}_+^2 \rightarrow \mathbb{C}$ be a bounded function with the following decay

\begin{equation} \label{decay} |F(y_1, y_2) | \ll (y_1 y_2)^{2+\epsilon}, \end{equation}
as $y_1 \rightarrow 0$ or $y_2 \rightarrow 0$.
For such a function define its Lebedev-Kontorovich transform $F^{\#} : D \times D \rightarrow \mathbb{C}$ as

\begin{equation}
F^{\#}(\nu) = \iint \limits_{\mathbb{R}^2_+} F(y) W_{\nu}(y) \frac{dy_1 dy_2}{y_1 y_2},
\end{equation}
where $W_{\nu}(y)$ is the Whittaker function and $D \subset \mathbb{C}$ is of the form $(-\delta,\delta)\times i \mathbb{R}$ for some $\delta > 0$.

For $m,n \in \mathbb{Z}$, $m,n>0$, let $A_{\phi}(m,n)$ be the Fourier-Whittaker coefficient of a Maass form $\phi$ as in \cite{goldfeld}. We normalize the $GL(3)$ Hecke-Maass forms by choosing the first Fourier-Whittaker coefficient to be $A_{\phi}(1,1)=1$. Fix positive integers $m_1, m_2, n_1, n_2$.
The cuspidal contribution $\mathcal{C}$ is given by

\begin{equation}
\mathcal{C} = \sum \limits_j A_j(m_1,m_2) \overline{A_j(n_1,n_2)} \frac{\left|F^{\#}\left(\nu^{(j)}_1, \nu^{(j)}_2\right)\right|^2 }{6 \mathcal{L}_j \prod \limits_{k=1}^3 \Gamma\left(\frac{1+3\nu_k^{(j)}}{2}\right)\Gamma\left(\frac{1-3\nu_k^{(j)}}{2}\right)},
\end{equation}
with $j$ ranging over cuspidal Hecke-Maass forms on $SL(3,\mathbb{Z})$.
The minimal Eisenstein series is

\begin{equation}
\mathcal{E}_{min} = \frac{1}{(4 \pi i)^2} \int \limits_{-i \infty}^{i \infty} \int \limits_{-i \infty}^{i \infty} A_{\nu}(m_1,m_2) \overline{A_{\nu}(n_1,n_2)}\frac{\left|F^{\#}\left(\nu_1, \nu_2\right)\right|^2 }{\prod \limits_{k=1}^3 \left|\zeta \left(1+3\nu_k\right)\Gamma\left(\frac{1+3\nu_k}{2}\right)\right|^2} d\nu_1 d\nu_2
\end{equation}
where the Fourier coefficients satisfy $|A_{\nu}(n_1,n_2)| \ll_{\epsilon} (n_1 n_2)^{\epsilon}$.
The maximal Eisenstein series is

\begin{equation}
\mathcal{E}_{max} = \frac{c}{2 \pi i} \sum \limits_j \int \limits_{-i \infty}^{i \infty} \frac{B_{\nu,j}(m_1,m_2) \overline{B_{\nu,j}(n_1,n_2)} \left| F^{\#} \left(\nu - \frac{i r_j}{3}, \frac{2 i r_j}{3} \right) \right|^2}{L(1, \text{Ad} \ u_j) \left| L(1+3\nu, u_j)\right|^2 \left|\Gamma\left(\frac{1+3\nu-i r_j}{2}\right) \Gamma\left(\frac{1+2i r_j}{2}\right) \Gamma\left(\frac{1+3\nu+i r_j}{2}\right)\right|^2} \ d\nu,
\end{equation}
where $c$ is an absolute constant and $\left\{u_j \right\}$ is an orthogonal basis of Hecke-Maass forms on $SL(2,\mathbb{Z})$ with eigenvalues $1/4 + r_j^2$. The Fourier coefficients satisfy $|B_{\nu,j}(n_1,n_2)| \ll_{\epsilon} (n_1 n_2)^{1/2+\epsilon}$.
Set  
 \[\textbf{1}_{A | B} = \begin{cases} 1, & \mbox{if } A | B \\ 0, & \mbox{otherwise.} \end{cases}\]
 On the geometric side we have

\begin{equation}
\Sigma_1 = \textbf{1}_{\left\{ \substack{m_1 = n_1 \\ m_2 = n_2} \right\}}\iint \limits_{\mathbb{R}_2^+} \left|F(y_1,y_2)\right|^2 \frac{dy_1 dy_2}{(y_1 y_2)^3} = \textbf{1}_{\left\{ \substack{m_1 = n_1 \\ m_2 = n_2} \right\}} \left\langle F, F\right\rangle,
\end{equation}
\begin{equation}
\Sigma_{2a} = \sum \limits_{\delta = \pm 1} \sum \limits_{\substack{D_1 | D_2 \\ m_2 D_1^2 = n_1 D_2}} \frac{\tilde{S}(\delta m_1, n_1, n_2, D_1, D_2)}{D_1 D_2} \tilde{\mathcal{J}}_{\delta} \left(\sqrt{\frac{m_1 n_1 n_2}{D_1 D_2}}\right),
\end{equation}
\begin{equation}
\Sigma_{2b} = \sum \limits_{\delta = \pm 1} \sum \limits_{\substack{D_2 | D_1 \\ m_1 D_2^2 = n_2 D_1}} \frac{\tilde{S}(\delta m_2, n_2, n_1, D_2, D_1)}{D_1 D_2} \tilde{\mathcal{J}}_{\delta} \left(\sqrt{\frac{m_2 n_1 n_2}{D_1 D_2}}\right),
\end{equation}
\begin{equation}
\Sigma_{3} = \sum \limits_{\delta_1, \delta_2 = \pm 1} \sum \limits_{D_1, D_2} \frac{S(\delta_1 m_1, \delta_2 m_2, n_1, n_2, D_1, D_2)}{D_1 D_2} \mathcal{J}_{\delta_1, \delta_2} \left(\frac{\sqrt{m_1 n_2 D_1}}{D_2},\frac{\sqrt{m_2 n_1 D_2}}{D_1}\right),
\end{equation}
where

\begin{multline}
\tilde{\mathcal{J}}_{\delta}(A) = A^{-2} \iint \limits_{\mathbb{R}_+^2} \iint \limits_{\mathbb{R}^2} \overline{F(A y_1, y_2)} e(-\delta A x_1 y_1) F\left(y_2 \frac{\sqrt{1+x_1^2+x_2^2}}{1+x_1^2}, \frac{A}{y_1 y_2} \frac{\sqrt{1+x_1^2}}{1+x_1^2+x_2^2}\right) \\
\times e\left(y_2 \frac{x_1 x_2}{1+x_1^2} + \frac{A}{y_1 y_2} \frac{x_2}{1+x_1^2+x_2^2}\right) \frac{dx_1 dx_2 dy_1 dy_2}{y_1 y_2^2},
\end{multline}

\begin{multline}
\mathcal{J}_{\delta_1, \delta_2}(A_1, A_2) = (A_1 A_2)^{-2} \iint \limits_{\mathbb{R}_+^2} \iiint \limits_{\mathbb{R}^3} \overline{F(A_1 y_1, A_2 y_2)} e(-\delta_1 A_1 x_1 y_1 - \delta_2 A_2 x_2 y_2) \\ \times F\left(\frac{A_2}{y_2} \frac{\sqrt{(x_1 x_2 - x_3)^2+x_1^2+1}}{1+x_2^2+x_3^2}, \frac{A_1}{y_1} \frac{\sqrt{1+x_2^2+x_3^2}}{(x_1 x_2 - x_3)^2+x_1^2+1}\right) \\
\times e\left(-\frac{A_2}{y_2} \frac{x_1 x_3 + x_2}{1+x_2^2+x_3^2} - \frac{A_1}{y_1} \frac{x_2(x_1 x_2 - x_3) + x_1}{(x_1 x_2 - x_3)^2+x_1^2+1}\right) \frac{dx_1 dx_2 dx_3 dy_1 dy_2}{y_1 y_2},
\end{multline}

\begin{equation}
\tilde{S}(m_1, n_1, n_2, D_1, D_2) = \textbf{1}_{D_1 | D_2} \underset{\substack{C_1 (\bmod{D_1}),\ C_2(\bmod{D_2}) \\ (C_1, D_1) = 1 = (C_2, D_2/D_1)}}{\sum \sum} e\left(\frac{m_1 C_1 + n_1 \overline{C_1} C_2}{D_1}\right) e\left(\frac{n_2 \overline{C_2}}{D_2/D_1}\right),
\end{equation}
\begin{multline}
\tilde{S}(m_1, m_2, n_1, n_2, D_1, D_2) =\underset{\substack{B_1, C_1 (\bmod{D_1}) \\ B_2, C_2 (\bmod{D_2}) \\ (B_1, C_1, D_1) = 1 = (B_2, C_2, D_2) \\ B_1 B_2 + C_1 D_2 + C_2 D_1 \equiv 0 (\bmod{D_1 D_2})}}{\sum \sum \sum \sum} e\left(\frac{m_1 B_1 + n_1(Y_1 D_2 - Z_1 B_2)}{D_1}\right) \\
\times e\left(\frac{m_2 B_2 + n_2(Y_2 D_1 - Z_2 B_1)}{D_2}\right),
\end{multline}
and $Y_1, Y_2, Z_1, Z_2$ are such that

\begin{equation}
Y_1 B_1 + Z_1 C_1 \equiv 1 \pmod{D_1} \ \text{and} \ Y_2 B_2 + Z_2 C_2 \equiv 1 \pmod{D_2}.
\end{equation}

\newpage

\section{Bounding the inverse transform}

Let

\begin{equation}
F^{\#}_{T,R}(\nu_1, \nu_2) = \left( T^{\alpha_1^2/R^2} + T^{\alpha_2^2/R^2} + T^{\alpha_3^2/R^2} \right) e^{(\alpha_1^2 + \alpha_2^2 + \alpha_3^2)/T^2} \left(\prod \limits_{j=1}^3 \Gamma\left(\frac{2+R+3\nu_j}{4} \right) \Gamma\left(\frac{2+R-3\nu_j}{4} \right) \right).
\end{equation}
As $F^{\#}_{T,R}$ has enough exponential decay on a strip $|\Re(\nu_1)|, |\Re(\nu_2)| < \epsilon$ then by Lebedev-Whittaker inversion as in section 2.2 of \cite{goldfeld},

\begin{equation}
F_{T,R}(y) = \frac{1}{(\pi i)^2} \int \limits_{-i \infty}^{i \infty} \int \limits_{-i \infty}^{i \infty} F^{\#}_{T,R}(\nu) \overline{W_{\nu}(y)} \frac{d\nu_1 d\nu_2}{\prod \limits_{j=1}^3 \Gamma\left(\frac{3\nu_j}{2}\right) \Gamma\left(-\frac{3\nu_j}{2}\right) }.
\end{equation}
We also have the Parseval-type identity

\begin{equation}
\label{parseval}
\left\langle F_{T,R}, F_{T,R} \right\rangle = \iint \limits_{\mathbb{R}_2^+} \left|F_{T,R}(y_1,y_2)\right|^2 \frac{dy_1 dy_2}{(y_1 y_2)^3} = \frac{1}{(\pi i)^2} \int \limits_{-i \infty}^{i \infty}  \int \limits_{-i \infty}^{i \infty} \frac{\left| F_{T,R}^{\#}(\nu) \right|^2 d\nu_1 d\nu_2}{ \prod \limits_{j=1}^3 \Gamma \left( \frac{3\nu_j}{2} \right) \Gamma \left( \frac{-3\nu_j}{2} \right)} = \left\langle F_{T,R}^{\#}, F_{T,R}^{\#}\right\rangle.
\end{equation}
For proofs of these results refer to \cite{transform}.

\begin{proposition} Fix $C_1, C_2 > 0$, $R > 3\max(C_1,C_2)+6$ and $\epsilon > 0$. For any $y_1, y_2 > 0$, $T \gg 1$, we have
\begin{equation}
\label{mainbound}
\left|F_{T,R}(y)\right| \ll_{C_1, C_2, R, \epsilon} \ y_1 y_2 T^{3R/2 + 11/2 + C_1/2 + C_2/2 + \epsilon} \left(\frac{y_1}{T}\right)^{C_1} \left(\frac{y_2}{T}\right)^{C_2}.
\end{equation}
\end{proposition}

\begin{proof}
We start by writing out the representation of $W_{\nu}$ as an inverse Mellin transform \cite{stade},

\[ W_{\nu}(y) = \frac{y_1 y_2 \pi^{3/2}}{(2\pi i)^2} \int \limits_{(C_2)} \int \limits_{(C_1)} \frac{\prod \limits_{j=1}^3 \Gamma\left(\frac{s_1 + \alpha_j}{2}\right) \Gamma\left(\frac{s_2-\alpha_j}{2}\right)}{4 \pi^{s_1+s_2} \Gamma\left(\frac{s_1+s_2}{2}\right)} y_1^{-s_1} y_2^{-s_2} ds_1 ds_2 \]
for $C_1, C_2 > 0$.
Combining this with the Lebedev-Whittaker inverse transform of $F_{T,R}$, we observe that

\[ F_{T,R}(y) = \int \limits_{(0)} \int \limits_{(0)} \int \limits_{(C_2)} \int \limits_{(C_1)}  \frac{F^{\#}_{T,R}(\nu_1, \nu_2)}{\prod \limits_{j=1}^3 \Gamma\left(\frac{3\nu_j}{2}\right) \Gamma\left(-\frac{3\nu_j}{2}\right)} \frac{\prod \limits_{j=1}^3 \Gamma\left(\frac{s_1 - \alpha_j}{2}\right) \Gamma\left(\frac{s_2+\alpha_j}{2}\right)}{16 \pi^{s_1+s_2+5/2} \Gamma\left(\frac{s_1+s_2}{2}\right)} y_1^{1-s_1} y_2^{1-s_2} ds_1 ds_2 d\nu_1 d\nu_2. \]
Assume $2k < C_1 < 2k+2$ for some non-negative interger $k$ and pull $s_1$ from the contour $(C_1)$ to the contour $(-C_1)$. Then

\begin{equation}
F_{T,R}(y) = \mathcal{M} + \sum_{m=0}^k \sum_{l=1}^3 \mathcal{R}_{m,l},
\label{firstpull}
\end{equation}
where $\mathcal{M}$ is the main term and is given by the same expression as $F_{T,R}$ with the contour $(C_1)$ replaced by $(-C_1)$. Here $\mathcal{R}_{m,l}$ are the residues corresponding to the poles of the integrand at $s_1 = \alpha_l - 2m$ for $m = 0, \cdots, k$ and $l=1,2,3$. These residues are given by

\[ \mathcal{R}_{m,l} = C_{m,l} \int \limits_{(0)} \int \limits_{(0)} \int \limits_{(C_2)}  \frac{F^{\#}_{T,R}(\nu_1, \nu_2)}{\prod \limits_{j=1}^3 \Gamma\left(\frac{3\nu_j}{2}\right) \Gamma\left(-\frac{3\nu_j}{2}\right)} \frac{\prod \limits_{j \neq l} \Gamma\left(\frac{\alpha_l -2m- \alpha_j}{2}\right) \prod \limits_{j=1}^3\Gamma\left(\frac{s_2+\alpha_j}{2}\right)}{\pi^{\alpha_l-2m+s_2} \, \Gamma\left(\frac{\alpha_l-2m+s_2}{2}\right)} y_1^{1-\alpha_l+2m} y_2^{1-s_2} ds_2 d\nu_1 d\nu_2 \]
for some constant $C_{m,l}$. Note that there will be no contribution of residues of higher order poles, which could possibly show up when two of the $\alpha_i$ differ by an even integer. In this situation, at least one of $\pm 3\nu_1$, $\pm 3\nu_2$, $\pm 3\nu_3$ is equal to a non-positive even integer, making the integrand identically zero in that region.

Now assume $2k' < C_2 < 2k'+2$ for some non-negative interger $k'$. For each term in ($\ref{firstpull}$) shift variable $s_2$ from the contour $(C_2)$ to $(-C_2)$ to get,

\begin{equation}
F_{T,R}(y) = \mathcal{\tilde{M}} + \sum_{m'=0}^k \sum_{l'=1}^3 \mathcal{M}_{m',l'} + \sum_{m, m'=0}^{k, k'} \sum_{l=1}^3 \sum_{l' \neq l} \mathcal{R}_{m,l,m',l'}
\label{secondpull}
\end{equation}
where $\mathcal{\tilde{M}}$ is the main term, given by the same expression as $\mathcal{M}$ with the contour $(C_2)$ replaced by $(-C_2)$ and $\mathcal{M}_{m',l'}$ are the residues corresponding to the poles of the integrand of $\mathcal{M}$ at $s_2 = -\alpha_{l'} - 2m'$ for $m' = 0, \cdots, k'$. The terms $\mathcal{R}_{m,l,m',l'}$ are the residues corresponding to the poles of the integrand of $\mathcal{R}_{m,l}$ at $s_2 = -\alpha_{l'} - 2m'$ for $m' = 0, \cdots, k'$ and $l'\neq l$.
The residues $\mathcal{M}_{m',l'}$ are given by

\[ \mathcal{M}_{m',l'} = C'_{m',l'} \int \limits_{(0)} \int \limits_{(0)} \int \limits_{(-C_1)}  \frac{F^{\#}_{T,R}(\nu_1, \nu_2)}{\prod \limits_{j=1}^3 \Gamma\left(\frac{3\nu_j}{2}\right) \Gamma\left(-\frac{3\nu_j}{2}\right)} \frac{\prod \limits_{j=1}^3 \Gamma\left(\frac{s_1 - \alpha_j}{2}\right) \prod \limits_{j \neq l'}  \Gamma\left(\frac{-\alpha_{l'}-2m'+\alpha_j}{2}\right)}{\pi^{s_1-\alpha_{l'}-2m'} \, \Gamma\left(\frac{s_1-\alpha_{l'}-2m'}{2}\right)} y_1^{1-s_1} y_2^{1+\alpha_{l'}+2m'} ds_1 d\nu_1 d\nu_2 \]
and the residues $\mathcal{R}_{m,l,m',l'}$ are given by

\[ \mathcal{R}_{m,l,m',l'} = C_{m,l,m',l'} \int \limits_{(0)} \int \limits_{(0)} \frac{F^{\#}_{T,R}(\nu_1, \nu_2)}{\prod \limits_{j=1}^3 \Gamma\left(\frac{3\nu_j}{2}\right) \Gamma\left(-\frac{3\nu_j}{2}\right)} \frac{\prod \limits_{j \neq l} \Gamma\left(\frac{\alpha_l -2m- \alpha_j}{2}\right) \prod \limits_{j \neq l'} \Gamma\left(\frac{-\alpha_{l'}-2m'+\alpha_j}{2}\right)}{4 \pi^{\alpha_l-2m-\alpha_{l'}-2m'} \Gamma\left(\frac{\alpha_l-2m-\alpha_{l'}-2m'}{2}\right)} y_1^{1-\alpha_l+2m} y_2^{1+\alpha_{l'}+2m'} d\nu_1 d\nu_2. \]
We will now bound each term in (\ref{secondpull}).
Let $\nu_j = i t_j$ and $s_j = C_j + i u_j$. Note that, for $\Re(\nu_j)=0$, the first exponential term in the definition of $F^{\#}_{T,R}$ is bounded and the second one has exponential decay for $|t_j| > T^{1+\delta}$. Using Stirling's formula to estimate the Gamma factors we get

\[ |\mathcal{\tilde{M}}| \ll y_1^{1+C_1} y_2^{1+C_2} \iint \limits_{|t_1|, |t_2| < T^{1+\delta}} \ \iint \limits_{\mathbb{R}^2} \mathcal{P} \cdot \exp\left( \mathcal{E} \right) du_1 du_2 dt_1 dt_2,\]
where the exponential factor is given by

\[ \frac{4 \mathcal{E}}{\pi} = 3\sum_{j=1}^3 |t_j| + |u_1 + u_2| - \sum_{j=1}^3 |i u_1 - \alpha_j| - \sum_{j=1}^3 |i u_2 + \alpha_j|, \]
and the polynomial term is given by

\begin{multline}\mathcal{P} = \left( \prod_{j=1}^3 (1+|t_j|) \right)^{(R+2)/2} \left( 1+|u_1+u_2| \right)^{(1+C_1+C_2)/2}
\\ \times \left( \prod_{j=1}^3 (1+|i u_1 - \alpha_j|) \right)^{(-C_1- 1)/2} \left( \prod_{j=1}^3 (1+|i u_2 + \alpha_j|) \right)^{(-C_2- 1)/2}.
\end{multline}
We now show that the exponential factor is non-positive. As the exponential factor is invariant under cyclic permutations of $(t_1,t_2,t_3)$, we may assume, without loss of generality, that $t_1$ and $t_2$ have the same sign. Then $|\alpha_1| + |\alpha_3| = 3|t_1|+3|t_2|$. As $|u_1+u_2| \leq |i u_1 - \alpha_2| + |i u_2 + \alpha_2|$, we get

\[ \frac{4 \mathcal{E}}{\pi} \leq 3\sum_{j=1}^3 |t_j| - |i u_1 - \alpha_1| - |i u_1 - \alpha_3| - |i u_2 + \alpha_1| - |i u_2 + \alpha_3| \]

\[ \leq 6|t_1| + 6|t_2| - 2(|\alpha_1| + |\alpha_3|) = 0.\]

For either $|u_1| > 5 T^{1+\delta}$ or $|u_2| > 5 T^{1+\delta}$, the exponential factor is bounded above by $-T^{1+\delta}$ giving exponential decay to the integral. Integrating first over $u_1$, $u_2$ we get

\[ |\mathcal{M}| \ll y_1^{1+C_1} y_2^{1+C_2} \iiiint \limits_{\substack{|t_1|, |t_2| < T^{1+\delta} \\ |u_1|, |u_2| < 5T^{1+\delta} }} \mathcal{P} du_1 du_2 dt_1 dt_2 \ll y_1^{1+C_1} y_2^{1+C_2} \, T^{(3R+11-C_1-C_2)/2 + \epsilon}\]
by choosing $\delta$ appropriately.
To bound the residues $\mathcal{M}_{m',l'}$ we start by shifting variables $\nu_1$ and $\nu_2$ to contours $(B_1)$ and $(B_2)$, respectively, where $|B_1|, |B_2| < R/3$ and $B'_{j} < C_1$ for $j \neq l$, defining $B'_1 = 2B_1 + B_2$, $B'_2 = B_2 - B_1$, $B'_3 = -B_1 - 2B_2$, in a manner similar to the $\nu_j$. Note that the first exponential term is now bounded by $3T^{(\max{B'_1, B'_2, B'_3})^2/R^2} \leq T$. It follows that

\[ |\mathcal{M}_{m',l'}| \ll y_1^{1+C_1} y_2^{1+B'_{l'} + 2m'} T \iint \limits_{|t_1|, |t_2| < T^{1+\delta}} \ \int \limits_{\mathbb{R}} \mathcal{P} \cdot \exp\left( \mathcal{E} \right) du_1 dt_1 dt_2 \]
where the exponential factor is given by

\[ \frac{4 \mathcal{E}}{\pi} = 3\sum_{j=1}^3 |t_j| + |u_1 - \Im(\alpha_{l'})| - \sum_{j=1}^3 |u_1 - \Im(\alpha_j)| - \sum_{j \neq l'} |\Im(\alpha_j-\alpha_{l'})|, \]
and the polynomial term is given by

\begin{multline}\mathcal{P} = \left( \prod_{j=1}^3 (1+|t_j|) \right)^{(R+2)/2} \left( 1+|u_1-\Im(\alpha_{l'})| \right)^{(1+C_1+B'_{l'}+2m')/2} 
\\ \times \left( \prod_{j=1}^3 (1+|u_1 - \Im(\alpha_j)|) \right)^{(-C_1- 1)/2} \prod_{j \neq l'} \left( (1+|\Im(-\alpha_{l'} + \alpha_j)|)^{(B'_j - B'_{l'}-2m'- 1)/2} \right).
\end{multline}
The exponential factor is again non-positive as 

\[ \frac{4 \mathcal{E}}{\pi} \, \leq \, 3\sum_{j=1}^3 |t_j| - \sum_{j \neq l'} |u_1 - \Im(\alpha_j)| - \sum_{j \neq l'} |\Im(\alpha_j-\alpha_{l'})| \, \leq 0 \]
by using the triangle inequality on the second sum to get a difference of $\Im(\alpha_j)$.
We now pick $B'_{l'} = C_2 - 2m'$ and $B'_j < 0$ for $j \neq l'$ to get

\[ |\mathcal{M}_{m',l'}| \ll y_1^{1+C_1} y_2^{1+C_2} T \iiint \limits_{\substack{|t_1|, |t_2| < T^{1+\delta} \\ |u_1| < 10 T^{1+\delta} }}  \mathcal{P} du_1 dt_1 dt_2 \ll y_1^{1+C_1} y_2^{1+C_2} T^{(3R+11-C_1-C_2)/2 + \epsilon}. \]
To bound the residues $\mathcal{R}_{m,l,m',l'}$ we again shift variables $\nu_1$ and $\nu_2$ to contours $(B_1)$ and $(B_2)$, respectively, where $|B_1|, |B_2| < R/3$. To simplify notation, we will assume without loss of generality that $l=1$ and $l'=2$. We obtain

\[ |\mathcal{R}_{m,1,m',2}| \ll y_1^{1-B'_1+2m} y_2^{1+B'_2 + 2m'} T \iint \limits_{|t_1|, |t_2| < T^{1+\delta}} \mathcal{P} \cdot \exp\left( \mathcal{E} \right) dt_1 dt_2 \]
where the exponential factor $\mathcal{E}$ is given by

\[ \frac{4 \mathcal{E}}{\pi} = 3\sum_{j=1}^3 |t_j| - \sum \limits_{j \neq 1} |\Im(\alpha_j - \alpha_1) | - \sum \limits_{j \neq 2} |\Im(\alpha_j - \alpha_2) | + |\Im(\alpha_1-\alpha_2)| = 0 \]
and the polynomial term is given by

\[ \mathcal{P} = \left( \prod_{j=1}^3 (1+|t_j|) \right)^{(R+2)/2} (1+|t_1|)^{(3B_1-1)/2} (1+|t_2|)^{(-3B_2-2m'-1)/2} (1+|t_3|)^{(3B_3-2m-1)/2}, \]
where $B_3 = B_1 + B_2$.
Then pick $B'_1 = -C_1 + 2m$ and $B'_2 = C_2 - 2m'$ to get

\[ |\mathcal{R}_{m,1,m',2}| \ll y_1^{1+C_1} y_2^{1+C_2} \, T \iint \limits_{|t_1|, |t_2| < T^{1+\delta}} \mathcal{P} dt_1 dt_2 \ll y_1^{1+C_1} y_2^{1+C_2} \, T^{(3R+9-2C_1 - 2C_2 + 2m + 2m')/2 + \epsilon}  \]

\[ \ll y_1^{1+C_1} y_2^{1+C_2} \, T^{(3R+9-C_1 - C_2)/2 + \epsilon}. \]
Combining all the bounds we get the desired result.

\end{proof}

\begin{remark}
It follows from this proposition that $F_{T,R}$ satisfies the decay condition (\ref{decay}).
\end{remark}

\newpage

\section{Bounds for the Kloosterman terms}

We start by getting bounds for the Kloosterman integrals $\tilde{\mathcal{J}}_{\delta}(A)$ and $\mathcal{J}_{\delta_1, \delta_2}(A_1, A_2)$.
Break the integral $\tilde{\mathcal{J}}_{\delta}(A) = \tilde{\mathcal{J}}_1 + \tilde{\mathcal{J}}_2 + \tilde{\mathcal{J}}_3 + \tilde{\mathcal{J}}_4$ depending on whether $y_1$ and $y_2$ are smaller or greater than $1$.
To estimate $\tilde{\mathcal{J}}_1$, start by taking absolute values inside the integral

\[ |\tilde{\mathcal{J}}_1| \ll A^{-2} \int_0^1 \int_0^1 \iint \limits_{\mathbb{R}^2} \left|F_{T,R}(A y_1, y_2)\right| \left| F_{T,R}\left(y_2 \frac{\sqrt{1+x_1^2+x_2^2}}{1+x_1^2}, \frac{A}{y_1 y_2} \frac{\sqrt{1+x_1^2}}{1+x_1^2+x_2^2}\right) \right|\frac{dx_1 dx_2 dy_1 dy_2}{y_1 y_2^2}. \]
Use (\ref{mainbound}) with $C_1 = C_2 = 6 - \epsilon/2$ to bound the first instance of $F_{T,R}$ and $C_1 = \epsilon$, $C_2 = 6 - \epsilon$ to bound the second instance of $F_{T,R}$ in order to get

\[ |\tilde{\mathcal{J}}_1| \ll A^{12-3\epsilon/2} \, T^{3R+2+3\epsilon/2} \int_0^1 \int_0^1 \iint \limits_{\mathbb{R}^2} \frac{y_1^{-1+\epsilon/2} y_2^{-1+3\epsilon/2} (1+x_1^2)^{3-3\epsilon/2} }{(1+x_1^2+x_2^2)^{6-3\epsilon/2}} dx_1 dx_2 dy_1 dy_2.\]
It is clear that the integral in $y$ converges and the integral in $x$ also converges for small values of $\epsilon$. After a suitable redefinition of $\epsilon$ we get

\[ |\tilde{\mathcal{J}}_1| \ll_{R, \epsilon} A^{12-\epsilon} \, T^{3R+2+\epsilon}. \]
For the remaining three integrals the argument is almost identical. The only necessary changes are to pick $C_1 = 6-3\epsilon/2$ when $y_1 > 1$ and $C_2 = 6 - 5\epsilon/2$ when $y_2 > 1$, while bounding the first instance of $F_{T,R}$. Putting all the bounds together, we obtain

\begin{equation}
\label{kbound1}
|\tilde{\mathcal{J}}_{\delta}(A)| \ll_{R, \epsilon} A^{12-\epsilon} \, T^{3R+2+\epsilon}.
\end{equation}
To bound $\mathcal{J}_{\delta_1, \delta_2}(A_1, A_2)$, we also break it into $\mathcal{J}_{\epsilon_1, \epsilon_2}(A_1, A_2) = \mathcal{J}_1 + \mathcal{J}_2 + \mathcal{J}_3 + \mathcal{J}_4$ depending on whether $y_i$ is smaller or greater than $1$.
We will first bound $\mathcal{J}_1$ by taking absolute values and doing the change of variables that sends $x_3$ to $x_3 + x_1 x_2$. We have

\begin{multline}
|\mathcal{J}_1| \ll (A_1 A_2)^{-2} \int_0^1 \int_0^1 \iiint \limits_{\mathbb{R}^3} \left|F_{T,R}(A_1 y_1, A_2 y_2) \right| \\
\times \left| F_{T,R}\left(\frac{A_2}{y_2} \frac{\sqrt{x_3^2+x_1^2+1}}{1+x_2^2+x_3^2}, \frac{A_1}{y_1} \frac{\sqrt{1+x_2^2+x_3^2}}{x_3^2+x_1^2+1}\right) \right| \frac{dx_1 dx_2 dx_3 dy_1 dy_2}{y_1 y_2}.
\end{multline}
Then apply (\ref{mainbound}) with $C_1 = C_2 = 6-\epsilon$ to bound the first instance of $F_{T,R}$ and $C_1 = C_2 = 6 - 2\epsilon$ to bound the second instance of $F_{T,R}$ in order to get

\[ |\mathcal{J}_1| \ll (A_1 A_2)^{10-3\epsilon} \, T^{3R-1+2\epsilon} \int_0^1 \int_0^1 \iiint \limits_{\mathbb{R}^3} \frac{(y_1 y_2)^{-1+\epsilon} dx_1 dx_2 dx_3 dy_1 dy_2}{\left((1+x_2^2+x_3^2)(1+x_1^2+x_3^2)\right)^{3-\epsilon}}. \]
As the integral in $y$ is convergent and the $x$ integral is also convergent for small values of $\epsilon$ then, after redefining $\epsilon$

\[ |\mathcal{J}_1| \ll_{R, \epsilon} (A_1 A_2)^{10-\epsilon} \, T^{3R-1+\epsilon}. \]
For the remaining three integrals the same argument works by choosing $C_1 = 6-\epsilon/2$ when $y_1 > 1$ and $C_2 = 6-\epsilon/2$ when $y_2>1$, while bounding the second instance of $F_{T,R}$. Combining all four bounds, one obtains

\begin{equation}
\label{kbound2}
|\mathcal{J}_{\delta_1, \delta_2}(A_1, A_2)| \ll_{R, \epsilon}  (A_1 A_2)^{10-\epsilon} \, T^{3R-1+\epsilon}.
\end{equation}
We may now bound the Kloosterman terms $\Sigma_{2a}$, $\Sigma_{2b}$ and $\Sigma_3$. The only necessary bounds for the Kloosterman sums will be

\[ \tilde{S}(m_1, n_1, n_2, D_1, D_2) \ll_{\epsilon} (D_1 D_2)^{1+\epsilon} \]
\[ S(m_1, m_2, n_1, n_2, D_1, D_2) \ll_{\epsilon} (D_1 D_2)^{1+\epsilon}.\]
To bound $\Sigma_{2a}$ we use (\ref{kbound1}) together with the first bound for Kloosterman sums, to obtain

\[ |\Sigma_{2a}| \ll_{R, \epsilon} \sum_{D_2=1}^{\infty} \sum \limits_{\substack{D_1 | D_2 \\ m_2 D_1^2 = n_1 D_2}} \frac{|\tilde{S}(\delta m_1, n_1, n_2, D_1, D_2)|}{D_1 D_2} \left|\tilde{\mathcal{J}}_{\delta} \left(\sqrt{\frac{m_1 n_1 n_2}{D_1 D_2}}\right) \right| \]
\[ \ll T^{3R+2+\epsilon} \sum \limits_{D_1, D_2} \frac{(m_1 n_1 n_2)^6}{(D_1 D_2)^{6-3\epsilon/2}} \ll (m_1 n_1 n_2)^6 \, T^{3R+2+\epsilon}. \]
The bound for $\Sigma_{2b}$ is obtained in the same manner. In order to bound $\Sigma_3$ we use (\ref{kbound2}) together with the second bound for Kloosterman sums, to obtain

\[ |\Sigma_3| \ll_{R, \epsilon} \sum \limits_{D_1, D_2} \frac{|S(m_1, m_2, n_1, n_2, D_1, D_2)|}{D_1 D_2} \left|\mathcal{J}_{\delta_1, \delta_2} \left(\frac{\sqrt{m_1 n_2 D_1}}{D_2},\frac{\sqrt{m_2 n_1 D_2}}{D_1}\right) \right| \]
\[ \ll (m_1 m_2 n_1 n_2)^5 T^{3R-1+\epsilon} \sum \limits_{\delta_1, \delta_2 = \pm 1} \sum \limits_{D_1, D_2} \frac{1}{(D_1 D_2)^{5-3\epsilon/2}} \ll (m_1 m_2 n_1 n_2)^5 \, T^{3R-1+\epsilon}. \]
For future reference, we write down these bounds in the following proposition.

\begin{proposition}
\label{kloostbounds}
Fix $R > 50$, $T \gg 1$ and $\epsilon > 0$. We have the following bounds for the Kloosterman terms:

\[ |\Sigma_{2a}| \ll_{R, \epsilon}  (m_1 n_1 n_2)^6 \, T^{3R+2+\epsilon}; \]

\[ |\Sigma_{2b}| \ll_{R, \epsilon}  (m_2 n_1 n_2)^6 \, T^{3R+2+\epsilon}; \]

\[ |\Sigma_{3}| \ll_{R, \epsilon}  (m_1 m_2 n_1 n_2)^5 \, T^{3R-1+\epsilon}. \]

\end{proposition}

\newpage

\section{Bounds for the Eisenstein terms}

We start by obtaining a bound for the contribution to the Kuznetsov trace formula of the minimal Eisenstein series $\mathcal{E}_{min}$. We will require the de la Vall\'{e}e Poussin bound for the Riemann zeta function,

\[ |\zeta(1+it)| \gg \frac{1}{\log(2+|t|)}. \]
Using the previous bound and Stirling's formula for the Gamma factors, we get

\[ |\mathcal{E}_{min}| \ll \int \limits_{-i T^{1+\epsilon}}^{i T^{1+\epsilon}} \int \limits_{-i T^{1+\epsilon}}^{i T^{1+\epsilon}} (m_1 m_2 n_1 n_2)^{\epsilon} \, \frac{ \prod \limits_{k=1}^3 \left| \Gamma\left(\frac{2+R+3\nu_k}{4} \right) \Gamma\left(\frac{2+R-3\nu_k}{4} \right) \right|^2 }{\prod \limits_{k=1}^3 \left|\zeta \left(1+3\nu_k\right)\Gamma\left(\frac{1+3\nu_k}{2}\right)\right|^2} |d\nu_1 d\nu_2| \]

\[ \ll  (m_1 m_2 n_1 n_2)^{\epsilon} \int \limits_{-i T^{1+\epsilon}}^{i T^{1+\epsilon}} \int \limits_{-i T^{1+\epsilon}}^{i T^{1+\epsilon}} \prod \limits_{k=1}^3 \left( (1+|\nu_k|)^R \log(2+|\nu_k|)^2 \right) |d\nu_1 d\nu_2| \ll_{R, \epsilon}  (m_1 m_2 n_1 n_2)^{\epsilon} \, T^{3R+2+\epsilon}.\]
To bound the maximal Eisenstein series contribution $\mathcal{E}_{max}$ we require the following lower bounds for $L$-functions

\[ L(1, \text{Ad} \ u_j) \gg_{\epsilon} (1+|r_j|)^{-\epsilon} \]

\[ |L(1+3\nu, u_j)| \gg_{\epsilon} (1+|\nu|+|r_j|)^{-\epsilon} \]
where the eigenvalue of $u_j$ is $1/4+r_j^2$. These lower bounds can be found in \cite{lowerbound1} and \cite{lowerbound2}. It follows that

\[ |\mathcal{E}_{max}| \ll \sum \limits_{r_j < T^{1+\epsilon}}   \int \limits_{-i T^{1+\epsilon}}^{i T^{1+\epsilon}} (m_1 m_2 n_1 n_2)^{1/2+\epsilon} \, \frac{\left|  \Gamma\left(\frac{2+R+3\nu+i r_j}{4}\right) \right|^8 \left| \Gamma\left(\frac{2+R+2 i r_j}{4}\right)\right|^4 (1+|r_j|)^{\epsilon} (1+|\nu|+|r_j|)^{2\epsilon} }{\left|\Gamma\left(\frac{1+3\nu-i r_j}{2}\right) \Gamma\left(\frac{1+2i r_j}{2}\right) \Gamma\left(\frac{1+3\nu+i r_j}{2}\right)\right|^2} |d\nu| \]

\[ \ll (m_1 m_2 n_1 n_2)^{1/2+\epsilon} \sum \limits_{r_j < T^{1+\epsilon}}  \int \limits_{-i T^{1+\epsilon}}^{i T^{1+\epsilon}} (1+|r_j|)^{R+\epsilon} (1+|\nu|+|r_j|)^{2R+2\epsilon} |d\nu| \]
\[ \ll_{R, \epsilon} (m_1 m_2 n_1 n_2)^{1/2+\epsilon} \, T^{3R+3+\epsilon}. \]
For the last inequality, we use Weyl's law for $GL(2)$ Maass forms which tells us that

\[ \left| \big\{ \phi_j : r_j < T \big\} \right| \sim  cT^2 \]
for some constant $c > 0$.
In summary, we obtain the following proposition.

\begin{proposition}
\label{eisensteinbounds}
Fix $R > 50$, $T \gg 1$ and $\epsilon > 0$. We have the following bounds for the Eisenstein terms in the Kuznetsov trace formula:

\[ |\mathcal{E}_{min}| \ll_{R, \epsilon} (m_1 m_2 n_1 n_2)^{\epsilon} \, T^{3R+2+\epsilon}; \]

\[ |\mathcal{E}_{max}| \ll_{R, \epsilon} (m_1 m_2 n_1 n_2)^{1/2+\epsilon} \, T^{3R+3+\epsilon}. \]

\end{proposition}

\newpage

\section{Estimating the main term on the geometric side}
To estimate the main term on the geometric side, $\Sigma_1 = \left\langle F_{T,R}, F_{T,R}\right\rangle$, we use the Parseval-type identity (\ref{parseval}) which says that $\Sigma_1 = \left\langle F_{T,R}^{\#}, F_{T,R}^{\#}\right\rangle$. Hence

\[ \left\langle F_{T,R}^{\#}, F_{T,R}^{\#}\right\rangle = \frac{1}{(\pi i)^2} \int \limits_{-i \infty}^{i \infty}  \int \limits_{-i \infty}^{i \infty} \frac{\left| F_{T,R}^{\#}(\nu) \right|^2 d\nu_1 d\nu_2}{ \prod \limits_{j=1}^3 \Gamma \left( \frac{3\nu_j}{2} \right) \Gamma \left( \frac{-3\nu_j}{2} \right)} \]

\[ = \frac{27^{R+1} \pi}{64^R} \int \limits_{-\infty}^{\infty}  \int \limits_{-\infty}^{\infty} \left( \sum_{j=1}^3 T^{-\beta_j^2/R^2} \right)^2 \exp \left(-2\sum_{j=1}^3 \beta_j^2/T^2 \right) \prod_{j=1}^3 \bigg( |t_j|^{R+1} + O(|t_j|^R + 1) \bigg) dt_1 dt_2 \]
where $\nu_j = i t_j$ and $\alpha_j = i \beta_j$. The second equality follows from Stirling's approximation of the Gamma factors. Making a linear change of variables of integration from $t_1, t_2$ to $\beta_1 = t_3 - t_1 = 2t_1+t_2, \beta_2=t_2-t_1$ and using the symmetry of the integral in the $\beta_i$, one gets

\[ = \frac{9^{R+1} 4\pi}{64^R} \int \limits_{0}^{\infty}  \int \limits_{0}^{\beta_2} \left( \sum_{j=1}^3 T^{-\beta_j^2/R^2} \right)^2 \exp \left(-2\sum_{j=1}^3 \beta_j^2/T^2 \right) \prod_{j\neq k} \bigg( |\beta_j-\beta_k|^{R+1} + O(|\beta_j-\beta_k|^R + 1) \bigg) d\beta_1 d\beta_2 \]
\[ = \frac{9^{R+1} 8\pi}{32^R} \int \limits_{0}^{\infty}  \int \limits_{0}^{\beta_2} \left( \sum_{j=1}^3 T^{-\beta_j^2/R^2} \right)^2 \exp \left(-2\sum_{j=1}^3 \beta_j^2/T^2 \right) \bigg( \beta_2^{3R+3} + O(\beta_2^{3R+2} \beta_1  +\beta_2^{3R+2} + 1) \bigg) d\beta_1 d\beta_2.\]
We now multiply out the integrand to get
\[ M + error = \frac{9^{R+1} 8\pi}{32^R} \int \limits_{0}^{\infty}  \int \limits_{0}^{\beta_2} \exp \left(-2\log T \beta_1^2/R^2-2\sum_{j=1}^3 \beta_j^2/T^2 \right) \beta_2^{3R+3} d\beta_1 d\beta_2 + error\]
where each of the error terms is bounded (up to a constant) by either
\[ I_1 = \int \limits_{0}^{\infty}  \int \limits_{0}^{\infty} \exp \left(-c_1 \beta_1^2-c_2 \beta_2^2 \right) \beta_2^{3R+3} d\beta_1 d\beta_2\]
with $c_1 \gg 1/T^2$ and $c_2 \gg \log T$ or
\[ I_2(l_1, l_2) = \int \limits_{0}^{\infty}  \int \limits_{0}^{\infty} \exp \left(-c_1 \beta_1^2-c_2 \beta_2^2 \right) \beta_2^{l_2} \beta_1^{l_1} d\beta_1 d\beta_2\]
where $(l_1,l_2) = (1, 3R+2), (0,3R+2)$ or $(0,0)$, and $c_1, c_2 \gg 1/T^2$, $\max(c_1,c_2) \gg \log T$.
Upon scaling, we get that $I_1 = O(c_1^{-1/2} c_2^{-(3R+4)/2}) = O(T)$.
For the other terms we get the following bound
\[ I_2(l_1, l_2) \ll c_1^{-(l_1+1)/2} c_2^{-(l_2+1)/2} = O(T^{max(l_1,l_2)+1}) = O(T^{3R+3}).\]
For $c_1^2 = 2\log T/R^2+4/T^2$, $c_2^2 = 4/T^2$ we write
\[ M = \frac{9^{R+1} 8\pi}{32^R} \int \limits_{0}^{\infty}  \int \limits_{0}^{\beta_2} \exp \left(-c_1^2 \beta_1^2 - c_2^2 \beta_1 \beta_2 - c_2^2 \beta_2^2 \right) \beta_2^{3R+3} d\beta_1 d\beta_2.\]
Performing a change of variables on the innermost integral, we get
\[ M = \frac{9^{R+1} 8\pi}{32^R c_1} \int \limits_{0}^{\infty}  \int \limits_{\beta_2 \frac{c_2^2}{2c_1}}^{\beta_2 \frac{2+c_2^2}{2c_1}} \exp \left(-\beta_1^2 - \left(c_2^2 - \frac{c_2^4}{4c_1^2} \right) \beta_2^2 \right) \beta_2^{3R+3} d\beta_1 d\beta_2.\]
We then use the following bounds
\[ \int_0^{x} e^{-t^2} dt \ll x \]
\[ \int_x^{\infty} e^{-t^2} dt \ll \frac{e^{-x^2}}{x}\]
to obtain
\[ M = \frac{9^{R+1} 4\pi^{3/2}}{32^R c_1} \int \limits_{0}^{\infty}  \exp \left(-\left(c_2^2 - \frac{c_2^4}{4c_1^2} \right) \beta_2^2 \right) \beta_2^{3R+3} \left(1 + O\left(\beta_2 \frac{c_2^2}{2c_1} \right) \right) d\beta_2\]
\[ + \int \limits_{0}^{\infty}  \exp \left(-\left(c_2^2 + \frac{1}{c_1^2} \right) \beta_2^2 \right) \beta_2^{3R+2} O\left( \frac{1}{(1+c_2^2)} \right) d\beta_2\]
\[ = \frac{9^{R+1} 2 \pi^2}{32^R} \frac{\left(c_2^2 - \frac{c_2^4}{4c_1^2} \right)^{-(3R+4)/2}}{c_1} + O\left(\frac{c_2^2 \left(c_2^2 - \frac{c_2^4}{4c_1^2} \right)^{-(3R+5)/2}}{c_1^2} \right) + O\left(\frac{\left(c_2^2 + \frac{1}{c_1^2} \right)^{-(3R+3)/2}}{1+c_2^2} \right) \]
\[ = \frac{3^{3R+2} \pi^2 R}{2^{8R + 7/2}} \frac{T^{3R+4}}{\sqrt{\log T}} + O(T^{3R+3}).\]

Combining the bounds for the Kloosterman terms (\ref{kloostbounds}), the bounds for the Eisenstein terms (\ref{eisensteinbounds}) and the computation of the main term we get (\ref{mainthm}).

\newpage

\section*{Acknowledgments}

The author would like to thank his advisor Dorian Goldfeld for his guidance and helpful discussions, as well as for the comments on the earlier drafts of this work.
The author was partially supported by FCT doctoral grant SFRH/BD/68772/2010.



\end{document}